\documentclass[a4paper,12pt,reqno]{amsart}
\usepackage{amssymb,amsmath,amsfonts,latexsym} 
\usepackage{hyperref}
\usepackage{ifthen}
\usepackage{graphicx}
\usepackage{float}
\usepackage{mathrsfs}
\usepackage{comment}
\usepackage{enumerate}
\usepackage[foot]{amsaddr}
 \usepackage[left=25mm,right=25mm,top=25mm,bottom=25mm,paper=a4paper]{geometry}

\theoremstyle{plain}

\newtheorem*{thm*}{Theorem}
\numberwithin{equation}{section}

\newtheorem{thm}{Theorem}[section]
\newtheorem{cor}[thm]{Corollary}
\newtheorem{lem}[thm]{Lemma}
\newtheorem{prop}[thm]{Proposition}
\theoremstyle{definition}
\newtheorem{defn}[thm]{Definition}
\theoremstyle{remark}

\numberwithin{equation}{section}

\theoremstyle{definition}
\newcounter {own}
\def\theown {\thesection  .\arabic{own}}

\newenvironment{pf}[1][]{%
 \vskip 3mm
 \noindent
 \ifthenelse{\equal{#1}{}}%
  {{\slshape Proof. }}%
  {{\slshape #1.} }%
 }%
{\qed\bigskip}

\newcounter{alphabet}

\newcommand{\beas}{\begin{eqnarray*}}
\newcommand{\eeas}{\end{eqnarray*}}
\newcommand{\bes} {\begin{equation*}}
\newcommand{\ees} {\end{equation*}}
\newcommand{\be} {\begin{equation}}
\newcommand{\ee} {\end{equation}}
\newcommand{\bea} {\begin{eqnarray}}
\newcommand{\eea} {\end{eqnarray}}

\newcounter{minutes}
\divide\time by 60
\newcounter{hours}
\multiply\time by 60 \addtocounter{minutes}{-\time}

\begin{document}
 
\title{Adams type Dunkl Stein-Weiss inequality on Dunkl Morrey spaces on the real line }

\thanks{
File:~\jobname .tex,
          printed: \number\year-\number\month-\number\day,
          \thehours.\ifnum\theminutes<10{0}\fi\theminutes}

\author{Sourav Dutta}
 \address{Sourav Dutta, School of Mathematical and Statistical Sciences, Indian Institute of Technology Mandi, Kamand,
 Mandi, HP 175005, India \newline
Email: d23026@students.iitmandi.ac.in }

\author{Saswata Adhikari$^\dagger$}
\address{Saswata Adhikari, School of Mathematical and Statistical Sciences, Indian Institute of Technology Mandi, Kamand,
 Mandi, HP 175005, India. \newline
Email: saswata@iitmandi.ac.in \newline
$^\dagger$ {\tt Corresponding author}}





\begin{abstract}
In this paper, we study the weighted boundedness of the Dunkl fractional integral operator (i.e., Dunkl Stein-Weiss inequality) associated with the Dunkl operator on $\mathbb{R}$. Indeed, we obtain the Adams-type Dunkl Stein-Weiss inequality on Dunkl-Morrey spaces. Our result extends the classical Adams type Stein-Weiss inequality on Morrey space result to the Dunkl setting. Furthermore, we establish the weighted boundedness of the Dunkl fractional maximal function on Dunkl Morrey spaces. 
\end{abstract}

\maketitle
\pagestyle{myheadings}
\markboth{Sourav Dutta, Saswata Adhikari}{Adams type Dunkl Stein-Weiss inequality on Dunkl Morrey spaces on the real line}

\subjclass \textbf{Mathematics Subject Classification[2010]}  {Primary  42B10, 47G10 ; Secondary 42B20, 42B25}

\section{Introduction} \label{sec:intro}
Hardy and Littlewood initially studied the one-dimensional fractional integral operator in \cite{ghl1}. They defined the fractional integral operator $I_{\gamma}$ on $(0, \infty)$ as 
\begin{equation}
    I_{\gamma}f(x)=\int\limits_0^\infty \frac{f(y)}{|x-y|^{1-\gamma}}dy, ~~0<\gamma<1. \label{eq:eq1}
\end{equation}
Under appropriate conditions on $p, q, \text{and}~\gamma$, they investigated the $(L^p, L^q)$ boundedness of $I_{\gamma}$. More precisely, they established the following:
\begin{thm}
    If $f\in L^p(0,\infty)$, where $1<p<\frac{1}{\gamma}$ with $\frac{1}{q}=\frac{1}{p}-\gamma$, then
\begin{equation}
    \|I_\gamma f\|_{L^q(0, \infty)} \le C\|f\|_{L^p(0, \infty)}, \label{eq:eq2}
\end{equation}
where $C$ denotes a positive constant independent of $f$.
\end{thm}

The $n$-dimensional analogue of \eqref{eq:eq1} is given by
\begin{equation}    
I_{\gamma}f(x)=\int\limits_{\mathbb{R}^n} \frac{f(y)}{|x-y|^{n-\gamma}}dy, ~~0<\gamma<n. \label{eq:eq3}
\end{equation}
In \cite{sls}, Sobolev proved the following:
\begin{thm}
    If $1<p<\frac{n}{\gamma}$ and $f\in L^p(\mathbb{R}^n)$, with $\frac{1}{q}=\frac{1}{p}-\frac{\gamma}{n}$, then
\begin{equation}
    \|I_\gamma f\|_{L^q(\mathbb{R}^n)} \le C\|f\|_{L^p({\mathbb{R}^n})}, \label{eq:eq4}
\end{equation}
where $C$ denotes a positive constant independent of $f$. This inequality is known as Hardy-Littlewood-Sobolev inequality.
\end{thm}
Therefore, it is natural to ask whether the inequalities \eqref{eq:eq2} and \eqref{eq:eq4} remain valid if the classical Lebesgue norm is replaced by the weighted Lebesgue norm 
\begin{eqnarray*}
    \|f\|_{p, \alpha}=\left(\int_{\mathbb{R}^n}|f(x)|^p|x|^{\alpha p}dx \right)^{\frac{1}{p}},~p>1,~\alpha \in \mathbb{R}.
\end{eqnarray*}

In \cite{ghl1}, Hardy and Littlewood established one such fundamental result for the one-dimensional case on $(0,\infty)$. Later, Stein and Weiss obtained such a weighted $(L^p, L^q)$ boundedness result of the fractional integral operator for the n-dimensional case in \cite{stw}, which is known as the Stein-Weiss inequality or the weighted Hardy-Littlewood-Sobolev inequality, stated as follows:
\begin{thm}
    Let $0<\gamma<n$, $1<p<\infty$, $\alpha<\frac{n}{p^\prime}$, $\beta<\frac{n}{q}$, $0\le\alpha+\beta\le\gamma$ and $\frac{1}{q}=\frac{1}{p}-\frac{\gamma-\alpha-\beta}{n}$. Then
\begin{eqnarray}
    \left(\int\limits_{\mathbb{R}^n}\{|I_{\gamma}f(x)||x|^{-\beta}\}^q~dx \right)^{1/q} \le C \left(\int\limits_{\mathbb{R}^n}\{|f(x)||x|^{\alpha}\}^{p}~dx \right)^{1/p}. \label{eq:eq5}
\end{eqnarray}
\end{thm}
Lieb \cite{lieb} determined the sharp constant for the Hardy-Littlewood-Sobolev inequality in the Euclidean setting. In the same paper, using the method of rearrangement technique and symmetrization, he also proved the existance of  extremals of \eqref{eq:eq5} in certain cases. The Hardy-Littlewood-Sobolev inequality on the Heisenberg group was formulated by Folland and Stein; refer to \cite{folland}, and the sharp constant in this scenario was derived by Frank and Lieb; see \cite{frank}. In \cite{guliyev}, Guliyev et al. derived the Stein-Weiss inequality for Carnot groups. Futhermore, the Hardy-Littlewood-Sobolev and Stein-Weiss inequalities were established on homogeneous Lie groups in \cite{kassy1,kassy2,ruzhan}. 

In 1938, Charles Morrey introduced Morrey spaces \cite{cbm} to investigate the local behavior of second-order elliptic partial differential equations. Over the years, these spaces have proven to be useful as powerful tools in fields such as potential theory, harmonic analysis and partial differential equations. 
For $1\leq p\leq q<\infty$, the (classical) Morrey space on $\mathbb{R}^{n}$ is defined by
\begin{eqnarray*}
L^{p,q}(\mathbb{R}^{n})=\{f\in L_{loc}^{p}(\mathbb{R}^{n}):\|f\|_{L^{p,q}}<\infty\},
\end{eqnarray*}
where $\|f\|_{L^{p,q}}=\sup\limits_{a\in\mathbb{R}^{n},r>0} r^{n(\frac{1}{q}-\frac{1}{p})}\bigg(\int\limits_{B(a,r)} |f(x)|^{p}dx\bigg)^{\frac{1}{p}}$. For $p=q$, $L^{p,q}(\mathbb{R}^{n})$ is the Lebesgue space $L^{p}(\mathbb{R}^{n})$. \\

 The boundedness of $I_{\gamma}$ on Morrey spaces was first studied by Spanne \cite{jp}. He proved the following:
\begin{thm}
Let $1<p<q<\frac{n}{\gamma},~\frac{1}{s}=\frac{1}{p}-\frac{\gamma}{n}$ and $\frac{1}{t}=\frac{1}{q}-\frac{\gamma}{n}$. Then
    \begin{eqnarray*}
    \|I_{\gamma}f\|_{L^{s,t}(\mathbb{R}^n)}\le C\|f\|_{L^{p,q}(\mathbb{R}^n)}, \forall f\in L^{p,q}(\mathbb{R}^n),
\end{eqnarray*}
where $C$ denotes a positive constant independent of $f$. 
\end{thm}
Later, Adams \cite {dra} reproved the boundedness of $I_{\gamma}$ on Morrey spaces and obtained a stronger result. He showed that
\begin{thm}
   Let $1<p\le q<\frac{n}{\gamma}$. Assume that the parameters $s$ and $t$ satisfy 
   \begin{eqnarray*}
       1<s\le t<\infty, \frac{1}{t}=\frac{1}{q}-\frac{\gamma}{n}, \frac{s}{t}=\frac{p}{q}.
   \end{eqnarray*}
   Then
   \begin{eqnarray*}
       \|I_{\gamma}f\|_{L^{s,t}(\mathbb{R}^n)}\le C\|f\|_{L^{p,q}(\mathbb{R}^n)}, \forall f\in L^{p,q}(\mathbb{R}^n),
   \end{eqnarray*}
  where $C$ denotes a positive constant independent of $f$.      
\end{thm}
In recent years, the weighted inequalities have been extensively studied on Morrey spaces. In \cite{tan}, Tanaka first carried out pioneering work in this direction. Motivated by this work, K.~P.~Ho studied the Stein-Weiss inequality on Morrey spaces in the Spanne range in \cite{kpho}, which is referred to as the Spanne type Stein-Weiss inequality. Kassymov et al. obtained the Stein-Weiss inequality in the Adams range on Morrey spaces associated with general homogeneous groups in \cite{kassy} and it is known as the Admas-type  Stein-Weiss inequality.\\

  Motivated by the above developments, it is natural to investigate whether the analogues of the Adams-type Stein-Weiss inequality remain valid on the Dunkl Morrey space.  In this paper, we restrict ourselves to the one-dimensional case. The Dunkl operators are differential-difference operators introduced and first studied by C.~F.~ Dunkl in the context of a theory of generalized spherical harmonics; for references, see \cite {cd1,cd2,cd3,cd4,cd5}. In \cite{cd5}, using the Dunkl kernel, Dunkl defined the Dunkl transform $\mathcal{F}_{\nu}$ which enjoys similar properties analogous to the classical Fourier transform. Several authors have studied the fundamental properties of the Dunkl transform; see, for instance, \cite{dejeu,cd5,ros2,ros3}. Given the fundamental role of the Fourier transform in analysis, it is quite rational to explore whether the results known in the Fourier setting remain valid for the Dunkl setting. Recently, in \cite{gorba}, Gorbachev et al. proved the Stein-Weiss inequality for the Dunkl fractional integral operator on the Lebesgue spaces within the Dunkl framework for any reflection group. Furthermore, they derived the sharp constant of the Stein-Weiss inequality, generalizing the classical results of Beckner \cite{beck} and Samko \cite{samko2}. In \cite{sa}, Adhikari et al. proved the existence of extremals of the Stein-Weiss inequality on the Lebesgue space for the Dunkl fractional integral operator in certain cases. \\

  We organize the paper as follows. In Section~\ref{sec:2}, we provide a brief introduction to Dunkl theory on the real line and some known results. In Section~\ref{sec3}, we prove the boundedness of the Dunkl fractional integral operator on Dunkl Morrey spaces. Finally, Section~\ref{sec4} is devoted to the proof of our main result, the Dunkl Stein--Weiss--Adams inequality on Dunkl Morrey spaces. Further, as a consequence of the above result, we obtain the weighted boundedness of the Dunkl fractional maximal function on Dunkl Morrey spaces.
  
  Finally, we mention that C represents a suitable positive constant that need not be the same in every occurrence.  

 \section{Preliminaries} \label{sec:2}
 Let $\nu\geq-\frac{1}{2}$ be a fixed real number. For $f\in C^{\infty}(\mathbb{R})$, the Dunkl operator associated with $\nu$ is defined by 
\begin{eqnarray*}
\Lambda_{\nu} f(x)=\frac{df}{dx} (x)+\frac{2\nu+1}{x} \frac{f(x)-f(-x)}{2},~x\in\mathbb{R}.
\end{eqnarray*}
Observe that $\Lambda_{-\frac{1}{2}}=\frac{d}{dx}$, indicating that the Dunkl derivative coincides with the classical derivative when $\nu=-\frac{1}{2}$.

For $\nu\ge -\frac{1}{2}, \lambda\in\mathbb{C}$, the initial value problem 
\begin{eqnarray*}
\Lambda_{\nu} f(x)=\lambda f(x), f(0)=1, 
\end{eqnarray*}
has a unique solution, denoted by $E_\nu(\lambda )$, which is called the Dunkl kernel \cite{dejeu,cd4,smsf} and is given by 
\begin{eqnarray*}
E_{\nu}(\lambda x)=j_{\nu}(i\lambda x)+\frac{\lambda x}{2(\nu+1)}j_{\nu+1}(i\lambda x),~x\in\mathbb{R},
\end{eqnarray*}
where $j_{\nu}$ is the normalized Bessel function of the first kind and of order $\nu$, defined by 
\begin{eqnarray*}
j_{\nu}(\lambda x)&=\left\{
\begin{array}{l l}
2^{\nu}\Gamma (\nu+1) \frac{J_{\nu}(\lambda x)}{(\lambda x)^{\nu}},& \quad \text{if~$\lambda x\neq 0$,}\\
1, & \quad \text{if~$\lambda x=0$,}
\end{array}\right.
\end{eqnarray*}
where $J_{\nu}$ is the Bessel function of the first kind and of order $\nu$. Note that $E_{-\frac{1}{2}}(\lambda x)=e^{\lambda x}$.

 The weighted Lebesgue measure on $\mathbb{R}$, denoted by $\mu_{\nu}$, is given by 
\begin{eqnarray*}
d\mu_{\nu}(x):=(2^{\nu+1}\Gamma(\nu+1))^{-1}|x|^{2\nu+1}dx.
\end{eqnarray*}

For $1\leq p\leq\infty$, the space $L^{p}(\mathbb{R},d\mu_\nu)$ denotes the set of all complex-valued measurable functions on $\mathbb{R}$ such that 
\begin{eqnarray*}
\|f\|_{L^{p}(\mathbb{R},d\mu_\nu)}:=\bigg(\int\limits_{\mathbb{R}}|f(x)|^{p}d\mu_{\nu}(x)\bigg)^{\frac{1}{p}}<\infty, if ~p~\in [1,\infty),
\end{eqnarray*}
and
\begin{eqnarray*}
\|f\|_{L^{\infty}(\mathbb{R},d\mu_\nu)}:
=\operatorname*{ess\,sup}_{x\in\mathbb{R}}|f(x)|<\infty,~ if~ p=\infty. 
\end{eqnarray*}

The Dunkl kernel corresponds to an integral transform known as the Dunkl transform which was studied in \cite{dejeu}. For $f\in L^{1}(\mathbb{R},d\mu_\nu)$, the Dunkl transform $\mathcal{F}_{\nu}f$ is defined by 
\begin{eqnarray*}
\mathcal{F}_{\nu} (f) (\lambda)=\int\limits_{\mathbb{R}} f(x) E_{\nu} (-i\lambda x)  d\mu_{\nu} (x),~\lambda\in\mathbb{R}.
\end{eqnarray*}
The above integral makes sense as $|E_{\nu}(ix)|\leq 1~\forall~x\in\mathbb{R}$ (see \cite{ros1}). Note that $\mathcal{F}_{-\frac{1}{2}}$ coincides with the classical Fourier transform $\mathcal{F}$ given by 
\begin{eqnarray*}
\mathcal{F} f(\lambda)= (2\pi)^{-\frac{1}{2}}\int\limits_{\mathbb{R}} f(x) e^{-i\lambda x} dx,~\lambda\in\mathbb{R}.
\end{eqnarray*}

The Dunkl transform satisfies the following properties:
\begin{thm}[\cite {trim}]
\leavevmode
\begin{itemize}
\item [(i)] For all $f\in L^{1}(\mathbb{R},d\mu_\nu)$, one has $\|\mathcal{F}_{\nu} (f)\|_{L^{\infty}(\mathbb{R},d\mu_\nu)}\leq \|f\|_{L^{1}(\mathbb{R},d\mu_\nu)}$.
\item [(ii)] For all $f\in\mathcal{S}(\mathbb{R})$, 
\begin{eqnarray*}
\mathcal{F}_{\nu} (\Lambda_\nu f)(\lambda)=i\lambda\mathcal{F}_{\nu} (f) (\lambda),~~\lambda\in\mathbb{R}.
\end{eqnarray*}
\item [(iii)] The Dunkl transform $\mathcal{F}_{\nu}$ is an isometric isomorphism from $L^{2}(\mathbb{R},d\mu_\nu)$ onto $L^{2}(\mathbb{R},d\mu_\nu)$. In particular, it satisfies the Plancherel formula i.e., 
\begin{eqnarray*}
\|\mathcal{F}_{\nu} (f)\|_{L^{2}(\mathbb{R},d\mu_\nu)}=\|f\|_{L^{2}(\mathbb{R},d\mu_\nu)}.
\end {eqnarray*}
\item [(iv)] If $f\in L^{1}(\mathbb{R},d\mu_\nu)$ with $\mathcal{F}_{\nu} (f)\in L^{1}(\mathbb{R},d\mu_\nu)$, then one has the inversion formula
\begin{eqnarray*}
f(x)=\int\limits_{\mathbb{R}} \mathcal{F}_{\nu} (f) (\lambda) E_{\nu}(i\lambda x) d\mu_\nu (\lambda)~a.e.~x\in\mathbb{R}.
\end{eqnarray*}
\end{itemize}
\end{thm}

Next, we consider the signed measure $\sigma_{x,y}$ on $\mathbb{R}$ given by 
\begin{eqnarray*}
d\sigma_{x,y}(z) &=\left\{
\begin{array}{l l }
W_\nu(x,y,z)d\mu_\nu(z),& \quad \text{if~$x,y\in\mathbb{R}\setminus \{0\}$,}\\
d\delta_x(z), & \quad \text{if~$y=0$,}\\
d\delta_y(z), & \quad \text{if~$ x=0$,}
\end{array}\right.
\end{eqnarray*}
where $W_\nu$ is an even function and it satisfies the following properties (see \cite{ros1}):
\begin{eqnarray*}
W_\nu(x,y,z) &=& W_{\nu} (y,x,z)=W_{\nu}(-z,y,-x)=W_{\nu} (-x,z,y),\\
&&\int\limits_{\mathbb{R}} |W_{\nu}(x,y,z)|d\mu_\nu(z)\leq 4.
\end{eqnarray*}
Also 
\begin{eqnarray*}
supp~(\sigma_{x,y})= S_{x,y}\cup (-S_{x,y})~with~S_{x,y}=[||x|-|y||,|x|+|y|].
\end{eqnarray*}

Let $x,y\in\mathbb{R}$ and $f$ be a continuous function on $\mathbb{R}$. Then the Dunkl translation operator of $f$ is given by 
\begin{eqnarray*}
\tau_x^\nu f(y)=\int\limits_{\mathbb{R}} f(z)d\sigma_{x,y}(z).
\end{eqnarray*}
The Dunkl translation operator $\tau^\nu_x$ satisfies the following properties:
\begin{prop} [\cite{mam}]
\leavevmode
\begin{itemize}
    \item [(i)] $\tau_x^\nu,~x\in\mathbb{R}$ is a bounded linear operator on $C^{\infty}(\mathbb{R})$.
    \item [(ii)] For all $f\in C^{\infty}(\mathbb{R})$ and $x,y\in\mathbb{R}$, one has 
    \begin{eqnarray*}
    \tau_x^{\nu} f(y)=\tau_y^{\nu} f(x),~\tau_0^{\nu} f(x)=f(x), \tau_x^\nu\circ\tau_y^\nu= \tau_y^\nu\circ\tau_x^\nu.
    \end{eqnarray*}
\end{itemize}
\end{prop}
If $f,g\in L^2(\mathbb{R}, d\mu_\nu)$, then
\begin{eqnarray*}
    \int\limits_{\mathbb{R}}\tau^\nu_x f(y)g(y)d\mu_\nu(y)=\int\limits_{\mathbb{R}} f(y)\tau^\nu_{-x}g(y)d\mu_\nu(y),~\forall x\in \mathbb{R}.
\end{eqnarray*}

If $f,g$ are continuous functions on $\mathbb{R}$ with compact support, then the generalized convolution of $f$ and $g$, denoted by $f\ast_{\nu} g$, is defined by
\begin{eqnarray*}
f\ast_{\nu} g (x)=\int\limits_{\mathbb{R}} f(y)\tau_{y}^{\nu} g(x)d\mu_{\nu} (y).
\end{eqnarray*}
The generalized convolution $\ast_{\nu}$ is associative and commutative \cite{ros1}. The following results are known.
\begin{thm}[\cite{sol}] \label{7pth3}
\leavevmode
\begin{itemize}
\item [(i)] For all $x\in\mathbb{R}$, the generalized translation operator $\tau_x^\nu$ extends to $L^{p}(\mathbb{R}, d\mu_\nu),p\geq 1$ and for every $f\in L^{p}(\mathbb{R}, d\mu_\nu)$, one has 
\begin{eqnarray*}
\|\tau_x^\nu f\|_{L^{p}(\mathbb{R},d\mu_\nu)}\leq 4\|f\|_{L^{p}(\mathbb{R},d\mu_\nu)}.
\end{eqnarray*}
\item [(ii)] If $f\in L^{1}(\mathbb{R}, d\mu_\nu)$, then 
\begin{eqnarray*}
\mathcal{F}_{\nu} (\tau_x^\nu f)(\lambda)=E_{\nu}(i\lambda x)\mathcal{F}_{\nu}(f)(\lambda),~x,\lambda\in\mathbb{R}.
\end{eqnarray*}
\item [(iii)] If  $f\in L^{1}(\mathbb{R}, d\mu_\nu)$ and $g\in L^{2}(\mathbb{R}, d\mu_\nu)$, then 
\begin{eqnarray*}
\mathcal{F}_{\nu}(f\ast_\nu g)=\mathcal {F}_{\nu} (f)\mathcal{F}_{\nu} (g).
\end{eqnarray*}
\end{itemize}
\end{thm}

In the sequel, we recall some definitions and known results that will be helpful throughout the paper. We begin with some important definitions from \cite{guliyev2009fractional}.

Let $0<\gamma<d_{\nu}$, where $d_{\nu}=2\nu+2$. For $f\in L_{loc}^1(\mathbb{R},d\mu_\nu)$ the Dunkl Hardy-Littlewood maximal function $M^\nu$ is defined as follows:
\begin{eqnarray*}
    M^\nu f(x)= \sup_{r>0}\frac{1}{\mu_\nu (B(0,r))}\int\limits_{B(0,r)}\tau^\nu_x|f|(y)d\mu_\nu(y).
\end{eqnarray*}
Now for $f\in L_{loc}^1(\mathbb{R},d\mu_\nu)$, the Dunkl fractional maximal function $M^\nu_\gamma$ is defined as follows:
\begin{eqnarray*}
    M^\nu_\gamma f(x)= \sup_{r>0}\frac{1}{\mu_\nu (B(0,r))^{1-\frac{\gamma}{d_\nu}}}\int\limits_{B(0,r)}\tau^\nu_x|f|(y)d\mu_\nu(y).
\end{eqnarray*}
Note that if $\gamma=0$, then $M^\nu_0$ coincides with the Dunkl Hardy-Littlewood maximal function $M^\nu$.\\
For every $f\in L^{p}_{loc}(\mathbb{R})$, $1\leq p<\infty$, the Dunkl fractional integral operator $I_{\gamma}^{\nu}$ is defined as follows:
\begin{eqnarray}
I_{\gamma}^{\nu}f(x)=f\ast_{\nu} K_{\gamma}^{\nu}(x)
=\int\limits_{\mathbb{R}} f(y)\tau_{y}^{\nu}K_{\gamma}^{\nu}(x)d\mu_{\nu}(y), \label{eq2.1}
\end{eqnarray}
where $K_{\gamma}^{\nu}(x)=|x|^{\gamma-d_{\nu}}$. The kernel $K_{\gamma}^{\nu}$ is known as the Dunkl Riesz kernel.

Let $B(x,r)=\{y\in\mathbb{R}: |y|\in ]\max \{0,|x|-r\}, |x|+r[\},~r>0$ and $b_\nu=[2^{\nu+1} (\nu+1)\Gamma (\nu+1)] ^{-1}$. Then $B(0,r)=]-r,r[$ and $\mu_\nu B(0,r)=b_\nu r^{2\nu+2}$.
\begin{defn}[\cite{sps}]
For $1\leq p\leq q< \infty$, the Dunkl Morrey space, denoted by $L^{p,q}(\mathbb{R},d\mu_\nu)$, is defined to be the set of all locally integrable functions $f$ on $\mathbb{R}$ such that 
\begin{eqnarray*}
\|f\|_{L^{p,q}(\mathbb{R},d\mu_\nu)}:=\sup\limits_{x\in\mathbb{R},r>0} r^{d_\nu(\frac{1}{q}-\frac{1}{p})}\bigg(\int\limits_{B(0,r)}\tau_{x}^{\nu}|f|^p(y)d\mu_{\nu}(y)\bigg)^{\frac{1}{p}}<\infty. 
\end{eqnarray*}
\end{defn}

\begin{lem}[Young's inequality {\cite{tsx}}] \label{lem3}
    Let $G=\mathbb{Z}_2$. Let $p,r,\xi\ge 1$ with $\frac{1}{r}+1=\frac{1}{p}+\frac{1}{\xi}$. Assume $f\in L^p(\mathbb{R},d\mu_\nu)$ and $g\in L^\xi(\mathbb{R},d\mu_\nu)$ respectively. Then
    \begin{eqnarray*}
\|f\ast_{\nu}g\|_{L^{r}(\mathbb{R},d\mu_\nu)}\leq C \|f\|_{L^{p}(\mathbb{R},d\mu_\nu)}\|g\|_{L^{\xi}(\mathbb{R},d\mu_\nu)}.
\end{eqnarray*}
\end{lem}

The following result has been proved on $\mathbb{R}^n$ in \cite{gorba}, which we will state here on $\mathbb{R}$.

\begin{lem}\label{7plem1}
The kernel $\Phi(x,y)=\tau_{x}^{\nu} K_{\gamma}^{\nu}(y)$ satisfies the following properties:
\begin{itemize}
    \item [(i)] $\Phi(x,y)=\Phi(y,x)$.
    \item [(ii)] $\Phi(rx,ty)=r^{\gamma-d_{\nu}}\Phi(x,\frac{ty}{r})$.
    \item [(iii)] $\Phi(x,y)=\int\limits_{\mathbb{R}}(x^{2}+y^{2}-2 x\eta)^{\frac{\gamma-d_{\nu}}{2}}d\mu_{y}^{\nu}(\eta)$,
\end{itemize}
where for each y in $\mathbb{R}$, $d\mu_{y}^{\nu}$ is a probability measure on $\mathbb{R}$, whose support is contained in $[-y,y]$.
\end{lem}

\begin{lem}[\cite{abdelkefi2007dunkl,guliyev2010p}] \label {7plem2}
Support of $\tau_x^\nu\chi_{B(0,r)}$ is contained in the ball $B(x,r)$ for all $x\in\mathbb{R}$.  Moreover,
\begin{eqnarray*}
0\leq\tau_{x}^{\nu}\chi_{B(0,r)}(y)\leq min\left\{1,\frac{2C_\nu}{2\nu+1}\bigg(\frac{r}{|x|}\bigg)^{2\nu+1}\right\},~\forall~y\in B(x,r), 
\end{eqnarray*}
with $C_\nu=\frac{\Gamma(\nu+1)}{\sqrt{\pi}\Gamma(\nu+\frac{1}{2})}$. In fact, when $|x|\leq r$, one has $0\leq\tau_{x}^{\nu}\chi_{B(0,r)}(y)\leq 1$,
and when $|x|>r$, one has $0\leq \tau_{x}^{\nu}\chi_{B(0,r)}(y)\leq \frac{2C_\nu}{2\nu+1}\big(\frac{r}{|x|}\big)^{2\nu+1}$.
\end{lem}

\begin{lem}[\cite{spsa}] \label{lem1}
Let $1<p \leq q<\infty.$ Then the Dunkl Hardy-Littlewood maximal function $M^{\nu}$ is bounded from $L^{p,q}(\mathbb{R},d\mu_\nu)$ to $L^{p,q}(\mathbb{R},d\mu_\nu)$ i.e.,
  \begin{eqnarray*}
      \|M^{\nu}f\|_{L^{p,q}(\mathbb{R},d\mu_{\nu})} \le C \|f\|_{L^{p,q}(\mathbb{R},d\mu_{\nu})},~\forall f\in L^{p,q}(\mathbb{R}),
  \end{eqnarray*}
where $C$ is a positive constant independent of $f$.  
\end{lem}

\section {The boundedness of Dunkl fractional integral operators on Dunkl Morrey spaces} \label{sec3}
The aim of this section is to prove the boundedness of the Dunkl fractional integral operator $I_\gamma^\nu$ on the Dunkl Morrey space in $\mathbb{R}$. This boundedness of $I^\nu_\gamma$ plays a crucial role in establishing the Dunkl Stein-Weiss-Adams inequality on Dunkl Morrey spaces in the subsequent section (that is, the Adams type Hardy-Littlewood-Sobolev inequality on Morrey spaces in Dunkl setting).

\begin{lem} \label{lem2}
Let $0<\gamma<d_\nu$ and $1<p\le q<\frac{d_\nu}{\gamma}$. Assume that the parameters $s$ and $t$ satisfy the relation 
$$1<s\leq t<\infty,~\frac{1}{t}=\frac{1}{q}-\frac{\gamma}{d_\nu} ~and~ \frac{s}{t}=\frac{p}{q}.$$ Then the Dunkl fractional integral operator $I_{\gamma}^{\nu}$ is bounded from $L^{p,q}(\mathbb{R},d\mu_\nu)$ to $L^{s,t}(\mathbb{R},\allowbreak d\mu_\nu)$ i.e., 
\begin{eqnarray*}
\|I_{\gamma}^{\nu}f\|_{L^{s,t}(\mathbb{R},d\mu_\nu)}\leq C \|f\|_{L^{p,q}(\mathbb{R},d\mu_\nu)}, \forall f\in L^{p,q}(\mathbb{R}, d\mu_\nu),
\end{eqnarray*}
where $C$ is a positive constant independent of $f$. 
\end{lem}
\begin{pf}
Let $r>0$ be given. Then, for any $f\in L^{p,q}(\mathbb{R},d\mu_\nu)$, we write from \eqref{eq2.1}
\begin{eqnarray}
    |I_\gamma^\nu f(x)|&=& \bigg|\int_{\mathbb{R}}\tau_x^\nu|y|^{\gamma-d_\nu}f(y)d\mu_{\nu}(y)\bigg|\nonumber\\
    &\le & \int\limits_{\mathbb{R}}\tau_x^\nu|y|^{\gamma-d_\nu}|f|(y)d\mu_{\nu}(y)\nonumber\\
    &=& \int\limits_{\mathbb{R}}\tau_{-x}^\nu|f|(y)|y|^{\gamma-d_\nu}d\mu_{\nu}(y)\nonumber\\
    &=& \int\limits_{B(0,r)}\tau_{-x}^\nu|f|(y)|y|^{\gamma-d_\nu}d\mu_{\nu}(y)+ \int\limits_{B^c(0,r)}\tau_{-x}^\nu|f|(y)|y|^{\gamma-d_\nu}d\mu_{\nu}(y)\nonumber\\
    &=& I_{1}+I_{2}, \label{eq:eq3.0}
\end{eqnarray}
where 
\begin{eqnarray*}
    I_{1}&:=&\int\limits_{B(0,r)}\tau_{-x}^\nu|f|(y)|y|^{\gamma-d_\nu}d\mu_{\nu}(y)
\end{eqnarray*}
and
\begin{eqnarray*} 
    I_{2}&:=&\int\limits_{B^c(0,r)}\tau_{-x}^\nu|f|(y)|y|^{\gamma-d_\nu}d\mu_{\nu}(y).
\end{eqnarray*}
Let us first estimate the integral $I_1$.
\begin{eqnarray}
    I_1&=& \int\limits_{B(0,r)}\tau_{-x}^\nu|f|(y)|y|^{\gamma-d_\nu}d\mu_{\nu}(y)\nonumber\\
         &=& \sum_{k=0}^{\infty}\int\limits_{\frac{r}{2^{k+1}}\le|y|<\frac{r}{2^k}}\tau_{-x}^\nu|f|(y)|y|^{\gamma-d_\nu}d\mu_{\nu}(y)\nonumber\\
         &\le& \sum\limits_{k=0}^{\infty} \bigg(\frac{r}{2^{k+1}}\bigg)^{\gamma-d_\nu} \int\limits_{B(0,\frac{r}{2^k})} \tau_{-x}^\nu|f|(y) d\mu_\nu(y)\nonumber\\
             &\le& C M^\nu f(-x) \sum\limits_{k=0}^{\infty} \bigg(\frac{r}{2^{k+1}}\bigg)^{\gamma-d_\nu} \bigg(\frac{r}{2^k}\bigg)^{d_\nu}\nonumber\\
         &\le& Cr^\gamma M^\nu f(-x) \sum\limits_{k=0}^{\infty}\frac{1}{2^{k\gamma}}\nonumber\\
         &{\le}&Cr^\gamma M^{\nu}f(-x),\label{eq:eq3.1}
\end{eqnarray}
as $\gamma>0$. Now we consider the integral $I_2$. Then
\begin{eqnarray}
    I_2&=& \int\limits_{B^c(0,r)} \tau_{-x}^\nu|f|(y)|y|^{\gamma-d_\nu}d\mu_\nu(y)\nonumber\\
    &=& \sum\limits_{k=0}^{\infty} \int\limits_{2^kr\le |y|<2^{k+1}r} \tau_{-x}^\nu|f|(y)|y|^{\gamma-d_\nu}d\mu_\nu(y)\nonumber\\
    &=& \sum\limits_{k=0}^{\infty} \int\limits_{\mathbb{R}} \tau_{-x}^\nu|f|(y)|y|^{\gamma-d_\nu}\chi_{B(0,2^{k+1}r)\setminus B(0,2^kr)}(y) d\mu_\nu(y)\nonumber\\
    &=& \sum\limits_{k=0}^{\infty} \int\limits_{\mathbb{R}} |f|(y)\bigg[\tau_{x}^\nu\bigg(|\cdot|^{\gamma-d_\nu}\chi_{B(0,2^{k+1}r)\setminus B(0,2^kr)}\bigg) (y)\bigg]^{\frac{1}{p}+\frac{1}{p^\prime}} d\mu_\nu(y)\nonumber\\
    &=& \sum\limits_{k=0}^{\infty} \int\limits_{\mathbb{R}} |f|(y)\bigg[\tau_{x}^\nu\bigg(|\cdot|^{\gamma-d_\nu}\chi_{B(0,2^{k+1}r)\setminus B(0,2^kr)}\bigg) (y)\bigg]^{\frac{1}{p}}\nonumber\\
    &&\times\bigg[\tau_{x}^\nu\bigg(|\cdot|^{\gamma-d_\nu}\chi_{B(0,2^{k+1}r)\setminus B(0,2^kr)}\bigg) (y)\bigg]^{\frac{1}{p^\prime}} d\mu_\nu(y)\nonumber\\
    &\le& \sum\limits_{k=0}^{\infty} \bigg[\int\limits_{\mathbb{R}} |f|^p(y)\tau_{x}^\nu\bigg(|\cdot|^{\gamma-d_\nu}\chi_{B(0,2^{k+1}r)\setminus B(0,2^kr)}\bigg)(y)d\mu_\nu(y) \bigg]^{\frac{1}{p}} \nonumber \\
    &&\times\bigg[\int\limits_{\mathbb{R}}\tau_{x}^\nu\bigg(|\cdot|^{\gamma-d_\nu}\chi_{B(0,2^{k+1}r)\setminus B(0,2^kr)}\bigg)(y)d\mu_\nu(y)\bigg]^{\frac{1}{p^\prime}},\nonumber
\end{eqnarray}
using H\"older's inequality with $\frac{1}{p}+\frac{1}{p^\prime}=1$. Therefore making use of Theorem~\ref{7pth3} and the assumption $\gamma<\frac{d_\nu}{q}$, we obtain
 \begin{eqnarray}   
   I_2&\le& \sum\limits_{k=0}^{\infty} \bigg(\int\limits_{\mathbb{R}} \tau_{-x}^\nu|f|^p(y)|y|^{\gamma-d_\nu}\chi_{B(0,2^{k+1}r)\setminus B(0,2^kr)}(y) d\mu_\nu(y)\bigg)^{\frac{1}{p}}\nonumber\\&&\times \bigg(\int\limits_\mathbb{R}|y|^{\gamma-d_\nu}\chi_{B(0,2^{k+1}r)\setminus B(0,2^kr)}(y)d\mu_\nu(y)\bigg)^{\frac{1}{p^{\prime}}}\nonumber\\
   &=& \sum\limits_{k=0}^\infty \bigg(\int\limits_{2^kr\le|y|<2^{k+1}r}\tau_{-x}^\nu|f|^p(y)|y|^{\gamma-d_\nu}d\mu_\nu(y)\bigg)^{\frac{1}{p}}\bigg(\int\limits_{2^kr\le|y|<2^{k+1}r}|y|^{\gamma-d_\nu}d\mu_\nu(y)\bigg)^{\frac{1}{p^\prime}}\nonumber\\
    &\le& \sum\limits_{k=0}^\infty (2^kr)^{\frac{\gamma-d_\nu}{p}}\bigg(\int\limits_{B(0,2^{k+1}r)}\tau_{-x}^\nu|f|^p(y)d\mu_{\nu}(y)\bigg)^{\frac{1}{p}} (2^kr)^{\frac{\gamma-d_\nu}{p^{\prime}}}\bigg(\int\limits_{B(0,2^{k+1}r)}d\mu_{\nu}(y)\bigg)^{\frac{1}{p^{\prime}}}\nonumber\\
    &\le& \sum\limits_{k=0}^{\infty}(2^kr)^{\gamma-d_\nu}\bigg(\int\limits_{B(0,2^{k+1}r)}\tau_{-x}^\nu|f|^p(y)d\mu_{\nu}(y)\bigg)^{\frac{1}{p}} \mu_{\nu}(B(0,2^{k+1}r))^{\frac{1}{p^{\prime}}}\nonumber\\
     &\le& C\|f\|_{{L^{p,q}}(\mathbb{R},d\mu_\nu)} \sum\limits_{k=0}^{\infty}(2^kr)^{\gamma-d_\nu}(2^{k+1}r)^{-{d_\nu(\frac{1}{q}-\frac{1}{p})}}(2^{k+1}r)^{\frac{d_\nu}{p^\prime}}\nonumber\\
    &=& Cr^{\gamma-\frac{d_\nu}{q}}\|f\|_{{L^{p,q}}(\mathbb{R},d\mu_\nu)}\sum\limits_{k=0}^{\infty}2^{k(\gamma-\frac{d_\nu}{q})}\nonumber\\
    &\le& Cr^{\gamma-\frac{d_\nu}{q}}\|f\|_{{L^{p,q}}(\mathbb{R},d\mu_\nu)}.\label{eq:eq3.2}    
\end{eqnarray}

Together from the last two estimates \eqref{eq:eq3.1} and \eqref{eq:eq3.2} and using \eqref{eq:eq3.0}, we obtain the following inequality:
$$|I^\nu_\gamma f(x)|\le C(r^\gamma M^\nu f(-x)+r^{\gamma-\frac{d_\nu}{q}}\|f\|_{{L^{p,q}}(\mathbb{R},d\mu_\nu)}).$$
Since the above inequality is valid for any $r>0$, by setting $r=\bigg(\frac{\|f\|_{{L^{p,q}}(\mathbb{R},d\mu_\nu)}}{M^\nu f(-x)}\bigg)^{\frac{q}{d_\nu}}$ and using the relation $\frac{1}{t}=\frac{1}{q}-\frac{\gamma}{d_\nu}$, we arrive at the estimate:
\begin{eqnarray}
    |I_{\gamma}^\nu f(x)|\le C\|f\|^{1-\frac{q}{t}}_{{L^{p,q}}(\mathbb{R},d\mu_\nu)}(M^{\nu}f(-x))^{\frac{q}{t}}.\label{eq:eq3.4}
\end{eqnarray}
Now for $x\in \mathbb{R}$ and $R>0$, we consider
\begin{eqnarray*}
    &&R^{d_\nu(\frac{1}{t}-\frac{1}{s})}\bigg(\int\limits_{B(0,R)}\tau^\nu_x|I_\gamma^\nu f|^s(y)d\mu_\nu(y)\bigg)^{\frac{1}{s}}\\
    &=& R^{d_\nu(\frac{1}{t}-\frac{1}{s})}\bigg(\int\limits_{\mathbb{R}}|I_\gamma^\nu f|^s(y)\tau^\nu_{-x}\chi_{B(0,R)}(y)d\mu_\nu(y)\bigg)^{\frac{1}{s}}\\
     \end{eqnarray*}
    \begin{eqnarray*}
     &{\le}& CR^{d_\nu(\frac{1}{t}-\frac{1}{s})}\|f\|^{1-\frac{q}{t}}_{{L^{p,q}}(\mathbb{R},d\mu_\nu)}\bigg(\int\limits_\mathbb{R}|M^\nu f|^{\frac{sq}{t}}(y)\tau^\nu_{-x}\chi_{B(0,R)}(y)d\mu_{\nu}(y)\bigg)^{\frac{1}{s}}\\
    &=& CR^{d_\nu(\frac{1}{t}-\frac{1}{s})}\|f\|^{1-\frac{q}{t}}_{{L^{p,q}}(\mathbb{R},d\mu_\nu)}\bigg(\int\limits_{B(0,R)}\tau^\nu_x|M^\nu f|^p(y)d\mu_{\nu}(y)\bigg)^{\frac{1}{s}},\\
    \end{eqnarray*}
    using \eqref{eq:eq3.4} and the relation $\frac{p}{q}=\frac{s}{t}$. Thus
    \begin{eqnarray*}
    &&R^{d_\nu(\frac{1}{t}-\frac{1}{s})}\bigg(\int\limits_{B(0,R)}\tau^\nu_x|I_\gamma^\nu f|^s(y)d\mu_\nu(y)\bigg)^{\frac{1}{s}}\\
    &\le& C\|f\|^{1-\frac{q}{t}}_{{L^{p,q}}(\mathbb{R},d\mu_\nu)}\bigg[R^{d_\nu(\frac{1}{t}-\frac{1}{s})\frac{t}{q}}\bigg(\int\limits_{B(0,R)}\tau^\nu_x|M^\nu f|^p(y)d\mu_{\nu}(y)\bigg)^{\frac{1}{p}}\bigg]^{\frac{q}{t}}\\
    &=& C\|f\|^{1-\frac{q}{t}}_{{L^{p,q}}(\mathbb{R},d\mu_\nu)}\bigg[R^{d_\nu(\frac{1}{q}-\frac{1}{p})}\bigg(\int\limits_{B(0,R)}\tau^\nu_x|M^\nu f|^p(y)d\mu_{\nu}(y)\bigg)^{\frac{1}{p}}\bigg]^{\frac{q}{t}}\\
    &\le& C\|f\|^{1-\frac{q}{t}}_{{L^{p,q}}(\mathbb{R},d\mu_\nu)}\|M^\nu f\|_{{L^{p,q}}(\mathbb{R},d\mu_\nu)}^{\frac{q}{t}}.
 \end{eqnarray*}   
Subsequently, using the boundedness of the Dunkl maximal function $M^\nu f$ from Lemma~\ref{lem1}, we conclude that

 \begin{eqnarray*}
      R^{d_\nu(\frac{1}{t}-\frac{1}{s})}\bigg(\int\limits_{B(0,R)}\tau^\nu_x|I_\gamma f|^s(y)d\mu_\nu(y)\bigg)^{\frac{1}{s}}
      &\le& C\|f\|^{1-\frac{q}{t}}_{{L^{p,q}}(\mathbb{R},d\mu_\nu)}\|f\|^{\frac{q}{t}}_{{L^{p,q}}(\mathbb{R},d\mu_\nu)}\\
    &=& C\|f\|_{{L^{p,q}}(\mathbb{R},d\mu_\nu)}.
 \end{eqnarray*}
 Finally, taking the supremum over $R>0$ and $x\in \mathbb{R}$ in the above inequality, we
complete the proof.
   
\end{pf}

\section {Adams type Dunkl Stein-Weiss inequality on Dunkl Morrey spaces} \label{sec4}
 In this section, we establish the Adams-type Dunkl Stein-Weiss inequality on the Dunkl Morrey space. Our idea of the proof is mainly based on \cite{kassy}.
 \begin{thm} \label{thm1}
     Let $0<\gamma <d_\nu,~ \alpha,\beta \in \mathbb{R},~0\le\alpha+\beta  \le \gamma<d_\nu,~1<p\le q< \frac{d_\nu}{\gamma-\alpha-\beta}$. Assume that the parameters $s$ and $t$ satisfy
     $$1<s\le t <\infty,~\frac{1}{t}=\frac{1}{q}-\frac{\gamma-\alpha-\beta}{d_\nu},~\frac{s}{t}=\frac{p}{q} ~\text{and}~\alpha<\frac{d_\nu}{q'},~\beta<\frac{d_\nu}{t}.$$
     In addition, if we assume that $d_\nu>t$, then
     $$\||\cdot|^{-\beta}I_\gamma^\nu f\|_{{L^{s,t}}(\mathbb{R},d\mu_\nu)}\le C\||\cdot|^\alpha f\|_{{L^{p,q}}(\mathbb{R},d\mu_\nu)},~\forall |\cdot|^\alpha f\in {{L^{p,q}}(\mathbb{R},d\mu_\nu)},$$
     where $C$ is a positive constant independent of $f$.

 \end{thm}

\begin{pf}
We note that when $s=t$, from the hypothesis condition $\frac{s}{t}=\frac{p}{q}$, it follows that $p=q$. Then the proof of Theorem~\ref{thm1} follows from the proof of Theorem~1.4 in \cite{gorba}. So, without loss of generality, we assume that $s<t$. Then $p<q$.
We write the Dunkl fractional integral operator $I_\gamma^\nu f$ as
\begin{eqnarray*}
    I_\gamma^\nu f(y)&=&\int\limits_\mathbb{R} \tau^\nu_y|z|^{\gamma-d_\nu}f(z)d\mu_\nu(z)\\
    &=&\int\limits_{|z|<\frac{|y|}{2}} \tau^\nu_y|z|^{\gamma-d_\nu}f(z)d\mu_\nu(z)+\int\limits_{\frac{|y|}{2}\le|z|\le2|y|} \tau^\nu_y|z|^{\gamma-d_\nu}f(z)d\mu_\nu(z)\\\
    &&+\int\limits_{|z|>2|y|} \tau^\nu_y|z|^{\gamma-d_\nu}f(z)d\mu_\nu(z)
\end{eqnarray*}
and hence, we have
\begin{eqnarray}
    |I_\gamma^\nu f(y)|&\le&\int\limits_{|z|<\frac{|y|}{2}} \tau^\nu_y|z|^{\gamma-d_\nu}|f(z)|d\mu_\nu(z)+\int\limits_{\frac{|y|}{2}\le|z|\le2|y|} \tau^\nu_y|z|^{\gamma-d_\nu}|f(z)|d\mu_\nu(z)\nonumber\\
    &&+\int\limits_{|z|>2|y|} \tau^\nu_y|z|^{\gamma-d_\nu}|f(z)|d\mu_\nu(z)\nonumber\\
    &=&J_1(y)+J_2(y)+J_3(y). \label{eq.eq5}
\end{eqnarray}
Now for $x\in\mathbb{R}$, $r>0$, we consider
\begin{eqnarray}
\int\limits_{B(0,r)}\tau_x^\nu(|\cdot|^{-\beta s}|I_\gamma^\nu f|^s)(y)d\mu_\nu(y)&=&\int\limits_{\mathbb{R}}\tau_x^\nu(|\cdot|^{-\beta s}|I_\gamma^\nu f|^s)(y)\chi_{B(0,r)}(y)d\mu_\nu(y)\nonumber\\
    &=&\int\limits_{\mathbb{R}}|y|^{-\beta s}|I^\nu_\gamma f|^s(y)\tau^\nu_{-x}\chi_{B(0,r)}(y)d\mu_\nu(y)\nonumber\\
    &\le&C\sum\limits_{i=1}^3\int\limits_{B(x,r)}|y|^{-\beta s}(J_i(y))^s\tau^\nu_{-x}\chi_{B(0,r)}(y)d\mu_\nu(y),\nonumber
    \\ \label{eq1}
\end{eqnarray}
using \eqref{eq.eq5}. The proof is organized into 4 steps:\\

\textbf{Step 1.} Consider the first term $J_1(y)=\int\limits_{|z|<\frac{|y|}{2}} \tau^\nu_y|z|^{\gamma-d_\nu}|f(z)|d\mu_\nu(z).$
Now using the representation of $\tau_y^\nu|z|^{\gamma-d_\nu}$ given in Lemma~\ref{7plem1}, we have
\begin{eqnarray}
    \tau_y^\nu|z|^{\gamma-d_\nu}=\int\limits_\mathbb{R}(y^2+z^2-2z\xi)^{\frac{\gamma-d_\nu}{2}}d\mu^\nu_y(\xi),\label{eq2}
\end{eqnarray}
where supp$~\mu^\nu_y \subset\{\xi \colon |\xi|\le|y|\}$.\\
Since $|z|<\frac{|y|}{2}$ and $|\xi|\le|y|$, it follows that
\begin{eqnarray*}
    \sqrt{y^2+z^2-2z\xi}\ge\sqrt{|y|^2+|z|^2-2|z||\xi|}\ge\sqrt{y^2+z^2-2|y||z|}=|y|-|z|\ge\frac{|y|}{2}.
\end{eqnarray*}
This implies that
\begin{eqnarray*}
    (y^2+z^2-2z\xi)^{\frac{\gamma-d_\nu}{2}}\le\bigg(\frac{|y|}{2}\bigg)^{\gamma-d_\nu}.
\end{eqnarray*}
Therefore
\begin{eqnarray}
    \tau_y^\nu|z|^{\gamma-d_\nu}&=&\int\limits_\mathbb{R}(y^2+z^2-2z\xi)^{\frac{\gamma-d_\nu}{2}}d\mu^\nu_y(\xi)\nonumber\\
    &\le&C(|y|)^{\gamma-d_\nu}\int\limits_\mathbb{R}d\mu_y^\nu(\xi)=C(|y|)^{\gamma-d_\nu}.\label{eq:e4.1}  
\end{eqnarray}
Now using \eqref{eq:e4.1}, we obtain
\begin{eqnarray}   
       J_1(y)&=&\int\limits_{|z|<\frac{|y|}{2}} \tau^\nu_y|z|^{\gamma-d_\nu}|f(z)|d\mu_\nu(z)\nonumber\\
    &\le&C(|y|)^{\gamma-d_\nu}\int\limits_{|z|<\frac{|y|}{2}}|f(z)|d\mu_\nu(z)\nonumber\\
    &\le&C(|y|)^{\gamma-d_\nu}\int\limits_{|z|<|y|}|f(z)|d\mu_\nu(z). \label{eq:eq4.2}
\end{eqnarray}
Let us estimate the above integral.
\begin{eqnarray}
    \int\limits_{|z|<|y|}|f(z)|d\mu_\nu(z)&=&\sum\limits_{k=0}^\infty \int\limits_{2^{-(k+1)}|y|\le|z|<2^{-k}|y|}|f(z)|d\mu_\nu(z)\nonumber\\
    &=&\sum\limits_{k=0}^\infty \int\limits_{2^{-(k+1)}|y|\le|z|<2^{-k}|y|}|f(z)|z|^\alpha|z|^{-\alpha}d\mu_\nu(z)\nonumber\\
    &\le&C\sum\limits_{k=0}^\infty (2^{-k}|y|)^{-\alpha}\int\limits_{2^{-(k+1)}|y|\le|z|<2^{-k}|y|}|z|^\alpha|f(z)|d\mu_\nu(z)\nonumber\\
    &\le&C\sum\limits_{k=0}^\infty (2^{-k}|y|)^{-\alpha}\int\limits_{|z|<2^{-k}|y|}|z|^\alpha|f(z)|d\mu_\nu(z)\nonumber\\
  &\le&C\sum\limits_{k=0}^\infty (2^{-k}|y|)^{-\alpha}\bigg(\int\limits_{|z|<2^{-k}|y|}|z|^{\alpha p}|f(z)|^pd\mu_\nu(z)\bigg)^{\frac{1}{p}}\nonumber\\
  &&\times(\mu_\nu(B(0,2^{-k}|y|)))^{\frac{1}{p^\prime}}\nonumber\\
    &\le& C\||\cdot|^\alpha f\|_{L^{p,q}({\mathbb{R}},d\mu_\nu)}\sum\limits_{k=0}^{\infty} (2^{-k}|y|)^{-\alpha+ \frac{d_\nu}{p'}+ \frac{d_\nu}{p}-\frac{d_\nu}{q}}\nonumber\\
    &=&C|y|^{-\alpha+d_\nu-\frac{d_\nu}{q}}\||\cdot|^\alpha f\|_{L^{p,q}({\mathbb{R}},d\mu_\nu)}\sum\limits_{k=0}^\infty (2^{-k})^{-\alpha+d_\nu-\frac{d_\nu}{q}}\nonumber\\
    &{\le}&C|y|^{-\alpha+d_\nu-\frac{d_\nu}{q}}\||\cdot|^\alpha f\|_{L^{p,q}({\mathbb{R}},d\mu_\nu)},\label{eq:eq4.3}       
\end{eqnarray}
where the above series converges since $\alpha<\frac{d_\nu}{q^\prime}$. Together with \eqref{eq:eq4.2} and \eqref{eq:eq4.3}, we obtain
    \begin{eqnarray}
        J_1(y)&\le& C|y|^{\gamma-d_\nu}|y|^{-\alpha+d_\nu-\frac{d_\nu}{q}}\||\cdot|^\alpha f\|_{L^{p,q}({\mathbb{R}},d\mu_\nu)}\nonumber\\
        &=& C|y|^{\gamma-\alpha-\frac{d_\nu}{q}}\||\cdot|^\alpha f\|_{L^{p,q}({\mathbb{R}},d\mu_\nu)}.\label{eq:eq4.4}
    \end{eqnarray}
Now we will estimate the right hand side integral of \eqref{eq1} for $i=1$:

\begin{eqnarray}
    &&\int\limits_{B(x,r)}|y|^{-\beta s}(J_1(y))^s\tau^\nu_{-x}\chi_{B(0,r)}(y)d\mu_\nu(y)\nonumber\\
    &\le& C\||\cdot|^\alpha f\|^s_{L^{p,q}({\mathbb{R}},d\mu_\nu)}\int\limits_{B(x,r)}|y|^{s(\gamma-\alpha-\beta-\frac{d_\nu}{q})}\tau^\nu_{-x}\chi_{B(0,r)}(y)d\mu_\nu(y) \nonumber\\
    &=&C\||\cdot|^\alpha f\|^s_{L^{p,q}({\mathbb{R}},d\mu_\nu)}\int\limits_{B(x,r)}|y|^{-\frac{s}{t}d\nu}\tau^\nu_{-x}\chi_{B(0,r)}(y)d\mu_\nu(y),\label{eq:eq4.5}
\end{eqnarray}
using the relation $\frac{1}{t}=\frac{1}{q}-\frac{\gamma-\alpha-\beta}{d_\nu}$ and \eqref{eq:eq4.4}.
We consider the following two cases to estimate the above integral.\\
\textbf{Case 1.}
Suppose that $|x|\le r$. Then $B(x,r)\subseteq B(0,2r).$ Using the Lemma~\ref{7plem2}, we obtain
\begin{eqnarray}
      \int\limits_{B(x,r)}|y|^{-\frac{s}{t}d\nu}\tau^\nu_{-x}\chi_{B(0,r)}(y)d\mu_\nu(y)
    &\le&\int\limits_{B(0,2r)}|y|^{-\frac{s}{t}d\nu}d\mu_\nu(y)\nonumber\\   
    &=&C\int\limits_0^{2r}y^{d_\nu(1-\frac{s}{t})-1}dy\nonumber\\
    &\le& Cr^{d_\nu(1-\frac{s}{t})}.\label{eq:eq4.6}   
\end{eqnarray}
\textbf{Case 2.}
Next, suppose $|x|>r$. Then $B(x,r)\subseteq B(0,2|x|)$. Again, making use of Lemma~\ref{7plem2} and the assumption $d_\nu>t$, we get
\begin{eqnarray}
    \int\limits_{B(x,r)}|y|^{-\frac{s}{t}d\nu}\tau^\nu_{-x}\chi_{B(0,r)}(y)d\mu_\nu(y)
     &\le&C\int\limits_{B(0,2|x|)}|y|^{-\frac{s}{t}d\nu}\bigg(\frac{r}{|x|}\bigg)^{d_\nu-1}|y|^{d_\nu-1}dy\nonumber\\
     &\le& C\bigg(\frac{r}{|x|}\bigg)^{d_\nu-1}\int\limits_0^{2|x|}y^{d_\nu(1-\frac{s}{t})-1}dy\nonumber\\
     &\le& C\bigg(\frac{r}{|x|}\bigg)^{d_\nu-1}(|x|)^{d_\nu(1-\frac{s}{t})}\nonumber\\
     &=&Cr^{d_\nu -1}(|x|)^{1-\frac{s}{t}d_\nu}\nonumber\\
     &\le& Cr^{d_\nu -1}r^{1-\frac{s}{t}d_\nu}\nonumber\\
     &=&Cr^{d_\nu(1-\frac{s}{t})}. \label{eq:eq4.7}   
\end{eqnarray}
Combining \eqref{eq:eq4.6} and \eqref{eq:eq4.7}, we get
\begin{eqnarray}
    \int\limits_{B(x,r)}|y|^{-\frac{s}{t}d\nu}\tau^\nu_{-x}\chi_{B(0,r)}(y)d\mu_\nu(y)\le Cr^{d_\nu(1-\frac{s}{t})}. \label{eq:eq4.8} 
\end{eqnarray}
Thus, from \eqref{eq:eq4.5} and \eqref{eq:eq4.8}, we conclude that
\begin{eqnarray}
    \int\limits_{B(x,r)}|y|^{-\beta s}(J_1(y))^s\tau^\nu_{-x}\chi_{B(0,r)}(y)d\mu_\nu(y)\le Cr^{d_\nu(1-\frac{s}{t})}\||\cdot|^\alpha f\|^s_{L^{p,q}({\mathbb{R}},d\mu_\nu)}. \label{eq:eq4.9}
\end{eqnarray}

\textbf{Step 2.} Here, we estimate the term
\begin{eqnarray}
    \int\limits_{B(x,r)}|y|^{-\beta s}(J_2(y))^s\tau^\nu_{-x}\chi_{B(0,r)}(y)d\mu_\nu(y). \label{eq3}
\end{eqnarray}
 First, consider the case  $\gamma>\alpha+\beta\ge 0$. We recall that $J_2$ is given by (from \eqref{eq.eq5}) $J_2(y)=\int\limits_{\frac{|y|}{2}\le|z|\le 2|y|}\tau_y^\nu |z|^{\gamma-d_\nu}|f(z)|d\mu_\nu(z)$, where the representation of $\tau_y^\nu |z|^{\gamma-d_\nu}$ is given in \eqref{eq2}. Therefore, when $\frac{|y|}{2}\le|z|\le2|y|$ with $|\xi|\le|y|$, we obtain
\begin{eqnarray*}
    (y^2+z^2-2z\xi)^{\frac{\alpha+\beta}{2}} \le (y^2+z^2+2|z||\xi|)^{\frac{\alpha+\beta}{2}}
                 &\le& (y^2+z^2+2|z||y|)^{\frac{\alpha+\beta}{2}}\\
                 &\le& (4|z|^2+|z|^2+2|z|2|z|)^{\frac{\alpha+\beta}{2}}\\
                 &\le& C|z|^{\alpha+\beta}\le C|z|^{\alpha}|y|^{\beta},
\end{eqnarray*}
and
\begin{eqnarray}
    \frac{(y^2+z^2-2z\xi)^{\frac{\gamma-d_\nu}{2}}}{|z|^{\alpha}|y|^{\beta}}&\le& C(y^2+z^2-2z\xi)^{-(\frac{\alpha+\beta}{2})+\frac{\gamma-d_\nu}{2}}\nonumber\\
    &=&C(y^2+z^2-2z\xi)^{\frac{\gamma-\alpha-\beta-d_\nu}{2}}\nonumber\\
    &=&C(y^2+z^2-2z\xi)^{\frac{\tilde{\gamma}-d_\nu}{2}}, \label{eq:eq4.14}
\end{eqnarray}
where $\tilde{\gamma}=\gamma-\alpha-\beta>0$. Hence, \eqref{eq2} and \eqref{eq:eq4.14} lead to
\begin{eqnarray}
    \frac{\tau^\nu_y|z|^{\gamma-d_\nu}}{|z|^{\alpha}|y|^{\beta}}&=&\int\limits_{\mathbb{R}}\frac{(y^2+z^2-2z\xi)^{\frac{\gamma-d_\nu}{2}}}{|z|^\alpha|y|^\beta}d\mu_y^\nu(\xi)\nonumber\\
    &\le&C\int\limits_{\mathbb{R}}(y^2+z^2-2z\xi)^{\frac{\tilde{\gamma}-d_\nu}{2}}d\mu_y^\nu(\xi)=C\tau^\nu_y|z|^{\tilde{\gamma}-d_\nu}. \label{eq:eq4.15}
\end{eqnarray}
Therefore, from \eqref{eq:eq4.15} we have
\begin{eqnarray}
    |y|^{-\beta}J_2(y)&=&|y|^{-\beta}\int\limits_{\frac{|y|}{2}\le|z|\le2|y|} \tau^\nu_y|z|^{\gamma-d_\nu}|f(z)|d\mu_\nu(z)\nonumber\\
    &=&\int\limits_{\frac{|y|}{2}\le|z|\le2|y|} \frac{\tau^\nu_y|z|^{\gamma-d_\nu}}{|z|^\alpha |y|^\beta}|z|^\alpha|f(z)|d\mu_\nu(z)\nonumber\\
    &\le&C\int\limits_{\frac{|y|}{2}\le|z|\le2|y|}\tau^\nu_y|z|^{\tilde{\gamma}-d_\nu}|\tilde{f}(z)|d\mu_\nu(z)\nonumber\\
    &\le&C\int\limits_{\mathbb{R}}\tau^\nu_y|z|^{\tilde{\gamma}-d_\nu}|\tilde{f}(z)|d\mu_\nu(z)=CI^\nu_{\tilde{\gamma}}|\tilde{f}|(y),  \label{eq4} 
\end{eqnarray}
where $\tilde{f}(z)=|z|^{\alpha}f(z)$. We remark here that the Dunkl fractional integral operator $I^\nu_{\tilde{\gamma}}$ is well defined since $0<\tilde{\gamma}<d_\nu$, which can be easily shown as follows:
\begin{eqnarray*}
    d_\nu>\gamma\ge \tilde{\gamma}=\gamma-\alpha-\beta>0.
\end{eqnarray*}
Furthermore, we observe from the assumptions of Theorem~4.1, $~\frac{1}{t}=\frac{1}{q}-\frac{\tilde{\gamma}}{d_\nu},~1<p\le q<\frac{d_\nu}{\tilde{\gamma}}$ and the parameters $s$ and $t$ that satisfy $1<s\le t<\infty$ with $\frac{s}{t}=\frac{p}{q}$. Therefore, using \eqref{eq4} and Lemma~\ref{lem2}, we obtain
\begin{eqnarray}
    \int\limits_{B(x,r)}|y|^{-\beta s}(J_2(y))^s\tau^\nu_{-x}\chi_{B(0,r)}(y)d\mu_\nu(y)&\le&C
    \int\limits_{B(x,r)}(I^\nu_{\tilde{\gamma}}(\tilde{|f|})(y))^s\tau^\nu_{-x}\chi_{B(0,r)}(y)d\mu_\nu(y)\nonumber\\
    &\le&C
    \int\limits_{\mathbb{R}}(I^\nu_{\tilde{\gamma}}(\tilde{|f|})(y))^s\tau^\nu_{-x}\chi_{B(0,r)}(y)d\mu_\nu(y)\nonumber\\
    &=&C\int\limits_{B(0,r)}\tau^\nu_x(I^\nu_{\tilde{\gamma}}(\tilde{|f|})^s(y)d_\nu(y)\nonumber\\
&\le&C{r^{d_\nu(1-\frac{s}{t})}}\|I^\nu_{\tilde{\gamma}}(|\tilde{f}|)\|^s_{L^{s,t}({\mathbb{R}},d\mu_\nu)}\nonumber\\
&\le&C{r^{d_\nu(1-\frac{s}{t})}}\|\tilde{f}\|^s_{L^{p,q}({\mathbb{R}},d\mu_\nu)}\nonumber\\
&=&Cr^{d_\nu(1-\frac{s}{t})}\||\cdot|^{\alpha}f\|^s_{L^{p,q}({\mathbb{R}},d\mu_\nu)}.\label{eq:eq4.17}
\end{eqnarray}
Let us now consider the case $\gamma=\alpha+\beta>0$, leading to $p=s$, $q=t$ and resulting the inequalities $1<p\le q<\infty$. Then \eqref{eq3} reduces to 
\begin{eqnarray*}
    \int\limits_{B(x,r)}|y|^{-\beta p}(J_2(y))^p\tau^\nu_{-x}\chi_{B(0,r)}(y)d\mu_\nu(y).
\end{eqnarray*}
In order to estimate the above integral, we consider the following two cases: $|x|\le r$ and $|x|>r$.\\
\textbf{Case 1.}
For $|x|\le r$, we have $B(x,r)\subseteq B(0,2r)$. Again, by using the Lemma~\ref{7plem2}, we compute
\begin{eqnarray}
    &&\int\limits_{B(x,r)}|y|^{-\beta p}(J_2(y))^p\tau^\nu_{-x}\chi_{B(0,r)}(y)d\mu_\nu(y)\nonumber\\
    &\le&\int\limits_{B(0,2r)}|y|^{-\beta p}(J_2(y))^pd\mu_\nu(y)\nonumber\\
    &=&\int\limits_{|y|\le 2r}|y|^{-\beta p}\bigg(\int\limits_{\frac{|y|}{2}\le|z|\le2|y|} \tau^\nu_y|z|^{\gamma-d_\nu}|f(z)|d\mu_\nu(z)\bigg)^pd\mu_\nu(y)\nonumber\\
    &=&\int\limits_{|y|\le 2r}|y|^{-\beta p}\bigg(\int\limits_{\frac{|y|}{2}\le|z|\le2|y|} \tau^\nu_y|z|^{\gamma-d_\nu}|z|^{-\alpha}|z|^{\alpha}|f(z)|d\mu_\nu(z)\bigg)^pd\mu_\nu(y)\nonumber\\
    &\le&C\int\limits_{|y|\le 2r}|y|^{-(\alpha+\beta)p}\bigg(\int\limits_{\frac{|y|}{2}\le|z|\le2|y|}\tau^\nu_y|z|^{\gamma-d_\nu}|\tilde{f}(z)|d\mu_\nu(z)\bigg)^pd\mu_\nu(y)\nonumber\\
    &=&C\sum\limits_{k=0}^\infty\int\limits_{\frac{2r}{2^{k+1}}<|y|\le\frac{2r}{2^k}}|y|^{-\gamma p}\bigg(\int\limits_{\frac{|y|}{2}\le|z|\le2|y|}\tau^\nu_y|z|^{\gamma-d_\nu}|\tilde{f}(z)|d\mu_\nu(z)\bigg)^pd\mu_\nu(y). \label{eq5}
\end{eqnarray}
For $\frac{|y|}{2}\le|z|\le2|y|$ and $\frac{2r}{2^{k+1}}<|y|\le\frac{2r}{2^k}$ we have
\begin{eqnarray}
    2^{-(k+1)}r\le |z|\le 2^{-k+2}r\\ [6pt]\label{eq:eq4.18}
    \text{and}\nonumber\\ [6pt]
    |z|\le2|y|\le3|y|\le3\cdot\frac{2r}{2^k}.\label{eq:eq4.22}
\end{eqnarray}
Thus, we get from \eqref{eq5}
\begin{eqnarray}
    &&\int\limits_{B(x,r)}|y|^{-\beta p}(J_2(y))^p\tau^\nu_{-x}\chi_{B(0,r)}(y)d\mu_\nu(y)\nonumber\\
    &\le&C\sum\limits_{k=0}^\infty\int\limits_{\frac{2r}{2^{k+1}}<|y|\le\frac{2r}{2^k}}|y|^{-\gamma p}\bigg(\int\limits_{2^{-(k+1)}r\le |z|\le 3\cdot\frac{2r}{2^k}}\tau^\nu_y|z|^{\gamma-d_\nu}|\tilde{f}(z)|d\mu_\nu(z)\bigg)^pd\mu_\nu(y)\nonumber\\
    &\le&C\sum\limits_{k=0}^\infty\bigg(\frac{2r}{2^{k+1}}\bigg)^{-\gamma p}\int\limits_{\frac{2r}{2^{k+1}}<|y|\le\frac{2r}{2^k}}\bigg(\int\limits_{2^{-(k+1)}r\le |z|\le 3\cdot\frac{2r}{2^k}}\tau^\nu_y|z|^{\gamma-d_\nu}|\tilde{f}(z)|d\mu_\nu(z)\bigg)^pd\mu_\nu(y)\nonumber\\
   &\le&C\sum\limits_{k=0}^\infty\bigg(\frac{2r}{2^{k+1}}\bigg)^{-\gamma p}
\int\limits_{\mathbb{R}}\bigg(\int\limits_{\mathbb{R}}\tau^\nu_y|z|^{\gamma-d_\nu}|\tilde{f}(z)|\chi_{\{2^{-(k+1)}r\le |z|\le 3\cdot\frac{2r}{2^k}\}}(z)d\mu_\nu(z)\bigg)^pd\mu_\nu(y)\nonumber\\
&=&C\sum\limits_{k=0}^\infty\bigg(\frac{2r}{2^{k+1}}\bigg)^{-\gamma p}\||\cdot|^{\gamma-d_\nu}\ast_k|\tilde{f}|\chi_{\{2^{-(k+1)}r\le |z|\le 3\cdot\frac{2r}{2^k}\}}\|^p_{L^p(\mathbb{R},d\mu_\nu)}\nonumber\\
&\le&C\sum\limits_{k=0}^\infty(2^{-k}r)^{-\gamma p}\||\cdot|^{\gamma-d_\nu}\chi_{|\cdot|\le3\cdot\frac{2r}{2^k}}\|^p_{L^1(\mathbb{R},d\mu_\nu)}\nonumber\\&&\times\|\tilde{f}\chi_{\{2^{-(k+1)}r\le |z|\le 3\cdot\frac{2r}{2^k}\}}\|^p_{L^p(\mathbb{R},d\mu_\nu)}, \label{eq:eq4.19}
\end{eqnarray}
by applying Young's inequality (see Lemma~\ref{lem3}) with $\xi=1$ and $p=r$.
We will now estimate the integral $\||\cdot|^{\gamma-d_\nu}\chi_{|\cdot|\le3\cdot\frac{2r}{2^k}}\|^p_{L^1(\mathbb{R},d\mu_\nu)}$.
\begin{eqnarray}
    \||\cdot|^{\gamma-d_\nu}\chi_{|\cdot|\le3\cdot\frac{2r}{2^k}}\|_{L^1(\mathbb{R},d\mu_\nu)}&=&\int\limits_{B(0,3\cdot\frac{2r}{2^k})}|x|^{\gamma-d_\nu}d\mu_\nu(x)\nonumber\\
    &=&\int\limits_{B(0,3\cdot\frac{2r}{2^k})}|x|^{\gamma-1}dx\le C(2^{-k}r)^\gamma. \label{eq:eq4.20}
\end{eqnarray}
Thus, combining \eqref{eq:eq4.19} and \eqref{eq:eq4.20} and using $p<q$, we obtain 
\begin{eqnarray}
    &&\int\limits_{B(x,r)}|y|^{-\beta p}(J_2(y))^p\tau^\nu_{-x}\chi_{B(0,r)}(y)d\mu_\nu(y)\nonumber\\
    &\le&C\sum_{k=0}^\infty(2^{-k}r)^{-\gamma p}(2^{-k}r)^{\gamma p}\|\tilde{f}\chi_{\{2^{-(k+1)}r\le |z|\le 3\cdot\frac{2r}{2^k}\}}\|^p_{L^p}\nonumber 
      \end{eqnarray}
    \begin{eqnarray}
    &\le&C\sum\limits_{k=0}^{\infty}\int\limits_{B(0,3\cdot\frac{2r}{2^k})}|\tilde{f}(z)|^pd\mu_\nu(z)\nonumber\\
    &\le&C\|\tilde{f}\|^p_{L^{p,q}}\sum\limits_{k=0}^\infty\bigg(\frac{r}{2^k}\bigg)^{d_\nu(1-\frac{p}{q})}\nonumber\\
    &{\le}&Cr^{d_\nu(1-\frac{p}{q})}\|\tilde{f}\|^p_{L^{p,q}({\mathbb{R}},d\mu_\nu)}\nonumber\\
    &\le&Cr^{d_\nu(1-\frac{p}{q})}\||\cdot|^\alpha f\|^p_{L^{p,q}({\mathbb{R}},d\mu_\nu)}\label{eq:eq4.21}.   
\end{eqnarray}
\textbf{Case 2.} For $|x|> r$, we have $B(x,r)\subseteq B(0,2|x|)$. Again, applying the Lemma~\ref{7plem2}, we compute
\begin{eqnarray}
     &&\int\limits_{B(x,r)}|y|^{-\beta p}(J_2(y))^p\tau^\nu_{-x}\chi_{B(0,r)}(y)d\mu_\nu(y)\nonumber\\
    &\le&C\bigg(\frac{r}{|x|}\bigg)^{d_\nu-1}\int\limits_{B(0,2|x|)}|y|^{-\beta p}(J_2(y))^pd\mu_\nu(y)\nonumber\\
    &\le&C\bigg(\frac{r}{|x|}\bigg)^{d_\nu-1}\int\limits_{B(0,2|x|)}|y|^{-\beta p}\bigg(\int\limits_{\frac{|y|}{2}\le|z|\le2|y|}\tau^\nu_y|z|^{\gamma-d_\nu}|f(z)|d\mu_\nu(z)\bigg)^pd\mu_\nu(y)\nonumber\\
    &\le&C\bigg(\frac{r}{|x|}\bigg)^{d_\nu-1}\int\limits_{|y|\le 2|x|}|y|^{-(\alpha+\beta)p}\bigg(\int\limits_{\frac{|y|}{2}\le|z|\le2|y|}\tau^\nu_y|z|^{\gamma-d_\nu}|\tilde{f}(z)|d\mu_\nu(z)\bigg)^pd\mu_\nu(y)\nonumber\\
    &=&C\bigg(\frac{r}{|x|}\bigg)^{d_\nu-1}\nonumber\\&&\times \sum_{k=0}^\infty\int\limits_{\frac{2|x|}{2^{k+1}}<|y|\le\frac{2|x|}{2^k}}|y|^{-\gamma p}\bigg(\int\limits_{\frac{|y|}{2}\le|z|\le2|y|}\tau^\nu_y|z|^{\gamma-d_\nu}|\tilde{f}(z)|d\mu_\nu(z)\bigg)^pd\mu_\nu(y). \label{eq:eq5.1}
    \end{eqnarray}
    Now for $\frac{|y|}{2}\le|z|\le2|y|$ and $\frac{2|x|}{2^{k+1}}<|y|\le\frac{2|x|}{2^k}$, we can derive $\frac{|x|}{2^{k+1}}\le |z|\le 3\cdot\frac{2|x|}{2^k}$, following the same steps as in  case~1 of step~2. Therefore, with $p<q$, \eqref{eq:eq5.1} leads to
    \begin{eqnarray}
    &&\int\limits_{B(x,r)}|y|^{-\beta p}(J_2(y))^p\tau^\nu_{-x}\chi_{B(0,r)}(y)d\mu_\nu(y)\nonumber\\
    &{\le}&C\bigg(\frac{r}{|x|}\bigg)^{d_\nu-1}\nonumber\\
    &&\times \sum_{k=0}^\infty\int\limits_{\frac{2|x|}{2^{k+1}}<|y|\le\frac{2|x|}{2^k}}|y|^{-\gamma p}\bigg(\int\limits_{\frac{|x|}{2^{k+1}}\le |z|\le 3\cdot\frac{2|x|}{2^k}}\tau^\nu_y|z|^{\gamma-d_\nu}|\tilde{f}(z)|d\mu_\nu(z)\bigg)^pd\mu_\nu(y)\nonumber
      \end{eqnarray}
    \begin{eqnarray}
    &=&C\bigg(\frac{r}{|x|}\bigg)^{d_\nu-1}\nonumber\sum\limits_{k=0}^\infty \bigg(\frac{2||x|}{2^{k+1}}\bigg)^{-\gamma p}\||\cdot|^{\gamma-d_\nu}\ast_k |\tilde{f}|\chi_{\{\frac{|x|}{2^{k+1}}\le |z|\le 3\cdot\frac{2|x|}{2^k}\}}\|^p_{L^p({\mathbb{R}},d\mu_\nu)}\nonumber\\
    &\le&C\bigg(\frac{r}{|x|}\bigg)^{d_\nu-1}\nonumber\sum\limits_{k=0}^\infty \bigg(\frac{2||x|}{2^{k+1}}\bigg)^{-\gamma p}\||\cdot|^{\gamma-d_\nu}\chi_{\{|\cdot|\le 3\cdot\frac{2|x|}{2^k}\}}\|^p_{L^1({\mathbb{R}},d\mu_\nu)}\\&&\times \||\tilde{f}\chi_{\frac{|x|}{2^{k+1}}\le|\cdot|\le3\cdot\frac{2|x|}{2^k}}\|_{L^p({\mathbb{R}},d\mu_\nu)}^p\nonumber\\
    &{\le}&C\bigg(\frac{r}{|x|}\bigg)^{d_\nu-1}\nonumber\sum\limits_{k=0}^\infty \bigg(\frac{|x|}{2^{k}}\bigg)^{-\gamma p}\bigg(\frac{|x|}{2^k}\bigg)^{\gamma p}\||\tilde{f}|\chi_{\frac{|x|}{2^{k+1}}\le|\cdot|\le3\cdot\frac{2|x|}{2^k}}\|_{L^p(\mathbb{R},d\mu_\nu)}^p\nonumber\\
    &\le&C\bigg(\frac{r}{|x|}\bigg)^{d_\nu-1}\sum\limits_{k=0}^\infty \int\limits_{B(0,3\frac{2|x|}{2^k})}|\tilde{f}(z)|^pd\mu_\nu(z)\nonumber \\
&\le&C\|\tilde{f}\|_{L^{p,q}(\mathbb{R},d\mu_\nu)}^p\bigg(\frac{r}{|x|}\bigg)^{d_\nu-1}\sum\limits_{k=0}^\infty\bigg(\frac{|x|}{2^k}\bigg)^{d_\nu(1-\frac{p}{q})}\nonumber\\
&{\le}&C\|\tilde{f}\|_{L^{p,q}(\mathbb{R},d\mu_\nu)}^p\bigg(\frac{r}{|x|}\bigg)^{d_\nu-1}(|x|)^{d_\nu(1-\frac{p}{q})}\nonumber\\
&=&C\|\tilde{f}\|^p_{L^{p,q}(\mathbb{R},d\mu_\nu)} r^{d_\nu-1} |x|^{1-d_\nu\frac{p}{q}}.\nonumber
\end{eqnarray}
Now our assumption $d_\nu>t$ and the fact $t=q$ when $\gamma=\alpha+\beta$ lead to $d_\nu>q$. This implies that $d_\nu > \frac{q}{p},~\forall\, p \ge 1.$ Therefore
\begin{eqnarray}
\int\limits_{B(x,r)}|y|^{-\beta p}(J_2(y))^p\tau^\nu_{-x}\chi_{B(0,r)}(y)d\mu_\nu(y)
&{\le}&C\|\tilde{f}\|_{L^{p,q}(\mathbb{R},d\mu_\nu)}^p r^{d_\nu-1}r^{1-d_\nu\frac{p}{q}}\nonumber\\
&=&Cr^{d_\nu(1-\frac{p}{q})}\|\tilde{f}\|^p_{L^{p,q}(\mathbb{R},d\mu_\nu)}.\nonumber\\
&=&Cr^{d_\nu(1-\frac{p}{q})}\||\cdot|^\alpha f\|^p_{L^{p,q}(\mathbb{R},d\mu_\nu)}. \label{eq:eq4.23}
\end{eqnarray}
Combining \eqref{eq:eq4.21} and \eqref{eq:eq4.23}, we have
\begin{eqnarray*}
   \int\limits_{B(x,r)}|y|^{-\beta p}(J_2(y))^p\tau^\nu_{-x}\chi_{B(0,r)}(y)d\mu_\nu(y)\le Cr^{d_\nu(1-\frac{p}{q})}\||\cdot|^\alpha f\|^p_{L^{p,q}(\mathbb{R},d\mu_\nu)}.\label{eq:eq4.24}
\end{eqnarray*}
Thus in step~2, for both consequences $\gamma>\alpha+\beta$ and $\gamma=\alpha+\beta$, we obtain
\begin{eqnarray}
   \int\limits_{B(x,r)}|y|^{-\beta s}(J_2(y))^s\tau^\nu_{-x}\chi_{B(0,r)}(y)d\mu_\nu(y)\le Cr^{d_\nu(1-\frac{s}{t})}\||\cdot|^\alpha f\|^s_{L^{p,q}(\mathbb{R},d\mu_\nu)}.\label{eq:eq5.2}
\end{eqnarray}

\textbf{Step 3.} Finally, we consider the last term $J_3(y)=\int\limits_{|z|>2|y|} \tau^\nu_y|z|^{\gamma-d_\nu}|f(z)|d\mu_\nu(z)$.\\
Since $|z|>2|y|$ and $|\xi|\le|y|$, it follows that
\begin{eqnarray*}
    \sqrt{y^2+z^2-2z\xi}\ge \sqrt{y^2+z^2-2|y||z|}=|z|-|y|>\frac{|z|}{2}.
\end{eqnarray*}
Again using Lemma~\ref{7plem1} and proceeding as in step 1, we obtain
\begin{eqnarray}
    \tau^\nu_y|z|^{\gamma-d_\nu}\le C|z|^{\gamma-d_\nu}.\label{eq:eq4.10}
\end{eqnarray}
Therefore, using \eqref{eq:eq4.10}, we have
\begin{eqnarray}
J_3(y)&\le&C\int\limits_{|z|>2|y|}|z|^{\gamma-d_\nu}|f(z)|d\mu_\nu(z)\nonumber\\    
&=& C\sum\limits_{k=1}^\infty \int\limits_{2^k|y|<|z|\le 2^{k+1}|y|}|z|^{\gamma-d_\nu-\alpha}|f(z)||z|^{\alpha}d\mu_\nu(z)\nonumber\\ 
&\le& C\sum\limits_{k=1}^\infty (2^{k+1}|y|)^{\gamma-d_\nu-\alpha}\int\limits_{|z|\le 2^{k+1}|y|}|f(z)||z|^{\alpha}d\mu_\nu(z)\nonumber\\
&\le& C\sum\limits_{k=1}^\infty (2^{k+1}|y|)^{\gamma-d_\nu-\alpha+\frac{d_\nu}{p^\prime}}\bigg(\int\limits_{|z|\le 2^{k+1}|y|}|f(z)|^p|z|^{\alpha p}d\mu_\nu(z)\bigg)^{\frac{1}{p}}\nonumber\\
&\le&C\||\cdot|^\alpha f\|_{L^{p,q}({\mathbb{R}},d\mu_\nu)}\sum\limits_{k=1}^\infty (2^{k+1}|y|)^{\gamma-d_\nu-\alpha+\frac{d_\nu}{p^\prime}-{d_\nu(\frac{1}{q}-\frac{1}{p})}}\nonumber\\
&=&C|y|^{\gamma-\alpha-\frac{d_\nu}{q}}\||\cdot|^\alpha f\|_{L^{p,q}({\mathbb{R}},d\mu_\nu)}\sum\limits_{k=1}^\infty 2^{(k+1)({\gamma-\alpha-\frac{d_\nu}{q}})}.\label{eq:eq4.11}
\end{eqnarray}
Using the relations $\frac{1}{t}=\frac{1}{q}-\frac{\gamma-\alpha-\beta}{d_\nu}$ and $\beta<\frac{d_\nu}{t}$, it easily follows that $\gamma-\alpha-\frac{d_\nu}{q}<0$ and therefore the above series converges.
Hence, from \eqref{eq:eq4.11}, it follows that
\begin{eqnarray}
    J_3(y)\le C|y|^{\gamma-\alpha-\frac{d_\nu}{q}}\||\cdot|^\alpha f\|_{L^{p,q}({\mathbb{R}},d\mu_\nu)}.\label{eq:eq4.12}
\end{eqnarray}
Now we will estimate the right hand side of \eqref{eq1} for $i=3$, using the relation $\frac{1}{t}=\frac{1}{q}-\frac{\gamma-\alpha-\beta}{d_\nu}$ and \eqref{eq:eq4.8} and \eqref{eq:eq4.12}. Thus, we obtain
\begin{eqnarray}
    &&\int\limits_{B(x,r)}|y|^{-\beta s}(J_3(y))^s\tau^\nu_{-x}\chi_{B(0,r)}(y)d\mu_\nu(y)\nonumber\\ 
    &\le&C\||\cdot|^\alpha f\|_{L^{p,q}({\mathbb{R}},d\mu_\nu)}^s\int\limits_{B(x,r)}|y|^{s(\gamma-\alpha-\beta-\frac{d_\nu}{q})}\tau_{-x}^\nu \chi_{B(0,r)}(y)d\mu_\nu(y)\nonumber\\
    &=& C\||\cdot|^\alpha f\|_{L^{p,q}({\mathbb{R}},d\mu_\nu)}^s\int\limits_{B(x,r)}|y|^{-\frac{s}{t}d_\nu}\tau_{-x}^\nu \chi_{B(0,r)}(y)d\mu_\nu(y)\nonumber\\
    &\le& Cr^{d_\nu(1-\frac{s}{t})}\||\cdot|^\alpha f\|^s_{L^{p,q}({\mathbb{R}},d\mu_\nu)}. \label{eq:eq4.13}  
\end{eqnarray}

\textbf{Step 4.} From \eqref{eq1}, \eqref{eq:eq4.9}, \eqref{eq:eq5.2} and \eqref{eq:eq4.13}, it follows that
\begin{eqnarray*}
     \int\limits_{B(0,r)}\tau_x^\nu(|\cdot|^{-\beta s}|I_\gamma^\nu f|^s)(y)d\mu_\nu(y) 
     &\le&C\sum\limits_{i=1}^3\int\limits_{B(x,r)}|y|^{-\beta s}(J_i(y))^s\tau^\nu_{-x}\chi_{B(0,r)}(y)d\mu_\nu(y)\\
     &\le& Cr^{d_\nu(1-\frac{s}{t})}\||\cdot|^\alpha f\|^s_{L^{p,q}({\mathbb{R}},d\mu_\nu)}.
\end{eqnarray*}
Finally, taking the supremum over $r>0$ and $x\in \mathbb{R}$ in the above inequality, we
complete the proof.
\end{pf}

Using Theorem~\ref{thm1}, we obtain the weighted boundedness of the Dunkl fractional maximal function $M^\nu_\gamma$ on the Dunkl Morrey space as follows:

\begin{cor}
 Let $0<\gamma <d_\nu,~ \alpha,\beta \in \mathbb{R},~0\le\alpha+\beta  \le \gamma<d_\nu,~1<p\le q< \frac{d_\nu}{\gamma-\alpha-\beta}$. Assume that the parameters $s$ and $t$ satisfy
     $$1<s\le t <\infty,~\frac{1}{t}=\frac{1}{q}-\frac{\gamma-\alpha-\beta}{d_\nu},~\frac{s}{t}=\frac{p}{q} ~\text{and}~\alpha<\frac{d_\nu}{q'},~\beta<\frac{d_\nu}{t}.$$
     In addition, if we assume that $d_\nu>t$, then
     $$\||\cdot|^{-\beta}M_\gamma^\nu f\|_{{L^{s,t}}(\mathbb{R},d\mu_\nu)}\le C\||\cdot|^\alpha f\|_{{L^{p,q}}(\mathbb{R},d\mu_\nu)},~\forall |\cdot|^\alpha f\in {{L^{p,q}}(\mathbb{R},d\mu_\nu)},$$
     where $C$ is a positive constant independent of $f$.
\end{cor}

\begin{pf}
    Let $r>0$ be given. Then
    \begin{eqnarray}
        \frac{1}{\mu_\nu(B(0,r))^{1-\frac{\gamma}{d_\nu}}}\int\limits_{B(0,r)}\tau^\nu_x|f|(y)d\mu_\nu(y)&=&\frac{1}{r^{d_\nu-\gamma}}\int\limits_{B(0,r)}\tau^\nu_x|f|(y)d\mu_\nu(y)\nonumber\\
        &\le&C\int\limits_{B(0,r)}\tau^\nu_x|f|(y)\frac{1}{|y|^{d_\nu-\gamma}}d\mu_\nu(y)\nonumber\\
        &\le&C\int\limits_{\mathbb{R}}\tau^\nu_x|f|(y)\frac{1}{|y|^{d_\nu-\gamma}}d\mu_\nu(y)\nonumber\\
        &=&C\int\limits_{\mathbb{R}}|f|(y)\tau^\nu_{-x}|y|^{\gamma-d_\nu}d\mu_\nu(y)=CI_\gamma^\nu|f|(-x).\nonumber
    \end{eqnarray}
  Finally, taking the supremum over $r>0$ in the above inequality, we have
  \begin{eqnarray}
      M^\nu_\gamma f(x)\le CI_\gamma^\nu|f|(-x), \forall x\in \mathbb{R}.\label{eq6.1}
  \end{eqnarray}
  Now for $x\in \mathbb{R}$ and $R>0$, we consider 
  \begin{eqnarray}
      &&R^{d_\nu(\frac{1}{t}-\frac{1}{s})}\bigg(\int\limits_{B(0,R)}\tau_x^\nu(|\cdot|^{-\beta s}|M_\gamma^\nu f|^s)(y)d\mu_\nu(y)\bigg)^{\frac{1}{s}}\nonumber\\
      &=&R^{d_\nu(\frac{1}{t}-\frac{1}{s})}\bigg(\int\limits_{\mathbb{R}}|y|^{-\beta s}|M^\nu_\gamma f|^s(y)\tau_{-x}^\nu \chi_{B(0,R)}(y)d\mu_\nu(y)\bigg)^{\frac{1}{s}}\nonumber\\
      &\le&R^{d_\nu(\frac{1}{t}-\frac{1}{s})}\bigg(\int\limits_{\mathbb{R}}|y|^{-\beta s}(I^\nu_\gamma |f|)^s(y)\tau_{-x}^\nu \chi_{B(0,R)}(y)d\mu_\nu(y)\bigg)^{\frac{1}{s}}\nonumber\\
      &=&R^{d_\nu(\frac{1}{t}-\frac{1}{s})}\bigg(\int\limits_{B(0,R)}\tau_x^\nu(|\cdot|^{-\beta s}(I_\gamma^\nu |f|)^s)(y)d\mu_\nu(y)\bigg)^{\frac{1}{s}}\nonumber\\
      &\le&\||\cdot|^{-\beta}I_\gamma^\nu |f|\|_{L^{s,t}(\mathbb{R},d\mu_\nu)}\nonumber\\
      &\le&\||\cdot|^\alpha f\|_{L^{p,q}(\mathbb{R},d\mu_\nu)},\nonumber
  \end{eqnarray}
  using the relation \eqref{eq6.1} and Theorem~\ref{thm1}. Therefore, taking the supremum over $R>0$ and $x\in \mathbb{R}$ in the above inequality, we complete the proof.
\end{pf}

  \bibliographystyle{amsplain}
  \bibliographystyle{abbrv}
\noindent\textbf{Data availability statement:} Not applicable.

\noindent\textbf{Conflicts of interest:} Both authors declare that they have no conflicts of interest.


\end{document}